\documentclass[a4paper,12pt]{amsart}

\usepackage{amscd}
\usepackage{amssymb, color}

\theoremstyle{plain}
  \newtheorem{thm}{Theorem}
  \newtheorem{prop}[thm]{Proposition}
  \newtheorem{lem}[thm]{Lemma}

  \newtheorem{thqt}{Theorem}

\theoremstyle{definition}

  \newtheorem{defi}{Definition}

\numberwithin{equation}{section}

\author[T. Yamazaki]{Takeaki Yamazaki}
\address{%
Department of Electrical, Electronic and Computer Engineering,
Toyo University,
Kawagoe-Shi, Saitama, 350-8585, Japan.
} \email{t-yamazaki@toyo.jp}

\keywords{Positive definite operator; operator mean;
operator monotone function; power mean;
Karcher equation; Karcher mean;
generalized Karcher equation;
relative operator entropy;
Tsalise relative operator entropy;
the Ando-Hiai inequality.}

\subjclass[2010]{Primary 47A63. Secondary 47A64.}

\title{The Ando-Hiai inequalities for 
the solution of the generalized Karcher equation
and related results}

\thanks{This research is supported by 
the INOUE ENRYO Memorial Grant, TOYO University.}
\begin{document}

\begin{abstract}
In this paper, we shall give a concrete relation
between the generalized Karcher equation and 
operator means as its solution.
Next, we shall show two types of the 
Ando-Hiai inequalities 
for the solution of the generalized Karcher equation. 
In this discussion, we also give a property of 
relative operator entropy.
\end{abstract}

\maketitle

\section{Introduction}
The theory of operator means was firstly considered 
in \cite{PW1975}. In that paper, the 
operator geometric mean has been defined.
Then the axiom of two-variable operator means
was introduced 
in \cite{KA1980}. 
However, this axiom cannot be extended to more than two operators, especially,
many people attempted to define 
operator geometric mean of $n$ operators
with natural properties
but they failed for a long time.
For this problem, 
the first solution was given in \cite{ALM2004}.
In that paper, a geometric mean of $n$ positive 
definite matrices with 10 nice properties was defined.
Since then operator geometric means 
have been discussed in many papers, for example,
\cite{BMP2010, IN2009, IN2010}.

Especially, we pay attention to a geometric mean of 
$n$ positive definite matrices which is
defined by a solution of a matrix equation  
in \cite{M2005}.
Then it was shown in \cite{BH2006} that 
this geometric mean can be 
defined by using the 
property that the set of all positive definite matrices 
is a Riemannian manifold with non-positive curvature.

For bounded linear operators on a 
Hilbert space case,
although, we can not define the geometric mean 
of $n$ operators in the same way as in \cite{BH2006},
it can be defined as a solution of the same operator
 equation as in \cite{M2005}, see \cite{LL2014}. 
This operator equation is called the Karcher equation, 
and the geometric mean is sometimes called 
the Karcher mean.
It is shown in \cite{BK2012, LL2011} 
that the Karcher mean satisfies all 10 properties 
stated in \cite{ALM2004}.
Moreover the Karcher mean satisfies 
the Ando-Hiai inequality 
-- one of the most important operator inequalities 
in the operator theory -- 
\cite{Y2012}, and 
the geometric mean 
which satisfies all 10 properties 
stated in \cite{ALM2004} and the Ando-Hiai 
inequality should be the Karcher mean \cite{Y2012}.
Hence we think that the Karcher mean is 
very important in operator theory.

As an extension of the Karcher mean, the power mean 
was defined in \cite{LP2012}.
 It interpolates the arithmetic, the 
geometric (Karcher) and 
the harmonic means, and it is defined by a solution of 
an operator equation. It is known that some operator 
inequalities which are 
relating to the power mean have been obtained in \cite{LY2013}.
In a recent year, P\'alfia \cite{P2016}
generalized the Karcher mean
by generalizing the Karcher equation.
Then he obtained various kinds of operator means of 
$n$ operators. 
We can obtain the Karcher and the power means 
as special cases of the new operator means.
But we have not known any concrete relation 
between the generalized Karcher equation and operator 
means, yet, i.e., we have not known 
which operator mean can be 
obtained from a given generalized Karcher equation.

In this paper, we shall give a concrete relation
between the generalized Karcher equation and 
operator means.
In fact, we will give an inverse function of a 
representing function of an operator mean which 
is derived from a given generalized Karcher equation.
In this discussion, the representing function of 
the relative operator entropy is very important.
Next, we shall give the Ando-Hiai type 
operator inequalities.
Here we shall show two types of Ando-Hiai inequalities,
and we shall give a property of the 
relative operator entropy. 
For the first type of the Ando-Hiai inequality, 
we shall give an Ando-Hiai type
operator inequality for a given operator mean.
The second one discusses an equivalence 
condition for operator means to
satisfy the Ando-Hiai type inequality.
The Ando-Hiai inequality was shown in \cite{AH1994},
firstly.
Then it has been extended to the Karcher and 
the power means  
as the first type of the Ando-Hiai inequality
in \cite{LL2014, LP2012, LY2013, Y2012}. 
On the other hand, the second type was firstly 
considered in \cite{W2014}. 
In that paper, an arbitrary operator mean 
of $2$ operators was considered.
In this paper, we shall generalize these results 
into several operator means of $n$ operators
which are derived from
the generalized Karcher equation.
At the same time, we shall study a property of the 
relative operator entropy.

This paper consists of the following:
In Section 2, we shall introduce some basic notations,
definitions and theorems which will be used in the paper.
In Section 3, we shall obtain a relation among 
the generalized Karcher equation, 
relative operator entropy and operator means.
In Section 4, we shall show the Ando-Hiai type 
inequalities for operator means which are 
derived from the solution of the 
generalized Karcher equation.
In the last section, we shall 
give a conjecture and a discussion.

\section{Preliminaries}
%In this section, we shall introduce some basic notations,
%definitions and theorems.
In what follows let $\mathcal{H}$ be a Hilbert space
with inner product $\langle \cdot, \cdot\rangle$, 
and $\mathcal{B(H)}$ be the set of all 
bounded linear operators 
on $\mathcal{H}$.
An operator $A\in \mathcal{B(H)}$ is positive semi-definite if $\langle Ax,x\rangle \geq 0$
holds for all $x\in \mathcal{H}$.
If $A$ is positive semi-definite, we denote 
$A\geq 0$. 
Let $\mathcal{PS}\subset 
\mathcal{B(H)}$ 
be the set of all positive semi-definite,
and let $\mathcal{P}\subset \mathcal{B(H)}$
be a cone of invertible positive semi-definite
 operators.
For self-adjoint operators $A$ and $ B$, 
$A\geq B$ is defined by $A-B\geq 0$.
%Any self-adjoint operator $A\in \mathcal{B(H)}$ 
%can be decomposed in the following form
%(spectral decomposition).
%
%$$ A=\int_{\sigma(A)}\lambda dE({\lambda}), $$
%
%where $\sigma(A)$ means the spectrum of $A$, and 
%$dE$ means a spectral measure on a real line.
%The for a real-valued function $f$ on 
%$\sigma(A)$,
%we define an operator function $f(A)$ as follows.
%
%$$ f(A)=\int_{\sigma(A)}f(\lambda) dE({\lambda}). $$
%
A real-valued function $f$ 
defined on an interval $I$ 
satisfying 
$$ B\leq A \ \Longrightarrow \ f(B)\leq f(A) $$
for all self-adjoint operators $A,B\in 
\mathcal{B(H)}$ such that
$\sigma(A), \sigma(B)\subset I$ 
is called an operator monotone function,
where $\sigma(X)$ denotes the spectrum of
$X\in \mathcal{B(H)}$.

\subsection{Operator mean}

\begin{defi}[Operator mean, \cite{KA1980}]
Let $\sigma: \mathcal{PS}^{2} \to \mathcal{PS}$ be
a binary operation. If $\sigma$ satisfies the 
following four conditions, then $\sigma$ is called
an operator mean.
\begin{itemize}
\item[(1)] If $A\leq C$ and $B\leq D$, then 
$\sigma(A,B)\leq \sigma(C,D)$,
\item[(2)] $X^{*}\sigma(A,B)X\leq 
\sigma(X^{*}AX, X^{*}BX)$ for all $X\in 
B(\mathcal{H})$,
\item[(3)] $A_{n}\downarrow A$ and 
$B_{n}\downarrow B$ imply $\sigma (A_{n}, B_{n})
\downarrow \sigma (A,B)$ in the 
strong operator topology,
\item[(4)] $\sigma(I,I)=I$, where $I$ means the 
identity operator in $\mathcal{B(H)}$.
\end{itemize}
\end{defi}

We notice that if $X$ is invertible in (2),
then equality holds.

\begin{thqt}[\cite{KA1980}]
Let $\sigma$ be an operator mean.
Then there exists a non-negative operator monotone 
function $f$ on $(0, \infty)$ such that $f(1)=1$ and 
$$ \sigma(A,B)=A^{\frac{1}{2}}f(A^{\frac{-1}{2}}
BA^{\frac{-1}{2}})A^{\frac{1}{2}} $$
for all $A\in \mathcal{P}$ and $B\in \mathcal{PS}$.
A function $f$ is called the representing 
function of an operator mean $\sigma$.
\end{thqt}

Especially, if the assumption $f(1)=1$ is removed, 
then $\sigma(A,B)$ is called a solidarity \cite{FFS1990} 
or a perspective \cite{EH2014}.
Let $\varepsilon$ be a positive real number. Then
we have $A_{\varepsilon}=A+\varepsilon I,
B_{\varepsilon}=B+\varepsilon I\in \mathcal{P}$
for $A,B\in \mathcal{PS}$, 
and we can compute an operator mean 
$\sigma(A,B)$ by 
$ \sigma(A,B)=\lim_{\varepsilon\searrow 0}
\sigma(A_{\varepsilon},B_{\varepsilon}). $
We note that for an operator mean 
$\sigma$ with a representing function $f$,
$f'(1)=\lambda
\in [0,1]$ (cf. \cite{H2013, ParXiv1208.5603}),
and we call $\sigma$ a $\lambda$-weighted operator
mean. 
Typical examples of operator means 
are the $\lambda$-weighted geometric  
and $\lambda$-weighted power means. These 
representing functions are $f(x)=x^{\lambda}$ and 
$f(x)=[1-\lambda+\lambda x^{t}]^{\frac{1}{t}}$, 
respectively, 
where $\lambda\in [0,1]$
and $t\in[-1,1]$ (in the case $t=0$, we consider 
$t\to 0$). The weighted power mean interpolates 
the arithmetic, the geometric and the harmonic means 
by putting $t=1,0,-1$, respectively.
In what follows, the $\lambda$-weighted geometric and 
$\lambda$-weighted power means 
of $A, B\in \mathcal{P}$ are denoted by
$A\sharp_{\lambda} B$ and 
$P_{t}(\lambda;A,B)$, respectively,
i.e., 
\begin{align*}
A\sharp_{\lambda}B  & 
= A^{\frac{1}{2}}(A^{\frac{-1}{2}}B
A^{\frac{-1}{2}})^{\lambda}A^{\frac{1}{2}},\\
 P_{t}(\lambda;A,B)  & =
A^{\frac{1}{2}}\left[1-\lambda+\lambda(A^{\frac{-1}{2}}B
A^{\frac{-1}{2}})^{t}\right]^{\frac{1}{t}}A^{\frac{1}{2}}.
\end{align*}

\subsection{The Karcher and the power means}

Geometric and power means of two operators 
can be extended 
to more than two operators
via the solution of operator equations as follows.
Let $n$ be a natural number, and let 
$\Delta_{n}$ be a set of all $n$-dimensional 
probability vectors, i.e., 
$$ \Delta_{n}=\{\omega=(w_{1},...,w_{n})\in (0,1)^{n}|\ 
\sum_{i=1}^{n}w_{i}=1\}. $$

\begin{defi}[The Karcher mean, 
\cite{BH2006, LL2014, M2005}]
Let $\mathbb{A}=(A_{1},...,A_{n})\in \mathcal{P}^{n}$ and 
$\omega=(w_{1},...,w_{n})\in \Delta_{n}$. Then 
the weighted Karcher mean $\Lambda(\omega;
\mathbb{A})$ is defined by a unique positive solution 
$X\in \mathcal{P}$ of 
the following operator equation;
$$ \sum_{i=1}^{n}w_{i}\log (X^{\frac{-1}{2}}A_{i}
X^{\frac{-1}{2}})=0. $$
\end{defi}
 
The Karchar mean of $2$ operators 
coincides with the geometric mean of $2$ operators,
i.e., for each $A,B\in \mathcal{P}$ and 
$\lambda \in [0,1]$, the solution of 
$$ (1-\lambda)\log (X^{\frac{-1}{2}}AX^{\frac{-1}{2}})+
\lambda\log (X^{\frac{-1}{2}}BX^{\frac{-1}{2}})=0 $$
is $X=A\sharp_{\lambda}B=A^{\frac{1}{2}}(A^{\frac{-1}{2}}BA^{\frac{-1}{2}})^{\lambda}
A^{\frac{1}{2}}$, as easily seen.
We can consider the Karcher mean as a geometric mean of 
$n$ operators.
Properties of the Karcher mean were shown in 
\cite{LL2014}.

The following power mean is an extension of the
Karcher mean which interpolates the arithmetic,
the harmonic and the Karcher (geometric) means.

\begin{defi}[The power mean, 
\cite{LL2014, LP2012}]
Let $\mathbb{A}=(A_{1},...,A_{n})\in \mathcal{P}^{n}$ and 
$\omega=(w_{1},...,w_{n})\in \Delta_{n}$. Then for 
$t\in [-1,1]$, the weighted power mean 
$P_{t}(\omega; \mathbb{A})$ is 
defined by a unique positive solution of 
the following operator equation;
$$ \sum_{i=1}^{n}w_{i}
(X^{\frac{-1}{2}}A_{i}X^{\frac{-1}{2}})^{t}=I. $$
\end{defi}

In fact, put $t=1$ and $t=-1$, then 
the arithmetic and harmonic means are easily obtained, 
respectively. Also let $t\to 0$. Then we have the 
Karcher mean \cite{LL2014, LP2012}.
Properties of the power mean were shown in 
\cite{LL2014, LP2012}.

Recently, the above operator equations are 
generalized as follows.
Let $\mathcal{M}$ be the set of all 
operator monotone functions on $(0,\infty)$, and let
$$ \mathcal{L}=\{ g\in \mathcal{M}|\ 
g(1)=0\text{ and }g'(1)=1\}.$$

\begin{defi}[Generalized Karcher Equation (GKE), 
\cite{P2016}]
Let $g\in \mathcal{L}$, $\mathbb{A}=
(A_{1},...,A_{n})\in \mathcal{P}^{n}$ and 
$\omega=(w_{1},...,w_{n})\in \Delta_{n}$. Then
the following operator equation is 
called the Generalized Karcher Equation
(GKE).
\begin{equation}
\sum_{i=1}^{n}w_{i}g(X^{\frac{-1}{2}}A_{i}
X^{\frac{-1}{2}})=0. 
\label{GKE}
\end{equation}
\end{defi}

\begin{thqt}[\cite{P2016}]
Any GKE has a unique positive solution 
$X\in \mathcal{P}$.
\end{thqt}

The Karcher and the power means can be obtained 
by putting $g(x)=\log x$ and 
$g(x)=\frac{x^t-1}{t}$ in \eqref{GKE}, respectively.
In what follows $\sigma_{g}
(\omega;\mathbb{A})$ (or $\sigma_{g}$, simply)
denotes the solution $X$ of \eqref{GKE}. 
Properties of $\sigma_{g}$ were obtained in 
\cite{P2016}, here we state some of them
as follows.

\begin{thqt}[\cite{P2016}]\label{prop:GKE mean}
Let $g\in \mathcal{L}$, $\omega\in \Delta_{n}$ and  
$\mathbb{A}=(A_{1},...,A_{n}),
\mathbb{B}=(B_{1},...,B_{n}) \in \mathcal{P}^{n}$.
Then $\sigma_{g}$ satisfies 
the following properties.
\begin{itemize}
\item[(1)] $\sigma_{g}(\omega; \mathbb{A})
\leq \sigma_{g}(\omega; \mathbb{B})$ holds
if $A_{i}\leq B_{i}$ for all $i=1,...,n$,
\item[(2)] $X^{*}\sigma_{g}(\omega; \mathbb{A})X=
\sigma_{g}(\omega; X^{*}\mathbb{A}X)$ for all 
invertible $X\in 
B(\mathcal{H})$, \\
where 
$X^{*}\mathbb{A}X=(X^{*}A_{1}X,...,X^{*}A_{n}X)$,
\item[(3)] $\sigma_{g}$ is continuous 
on $\mathcal{P}^n$, with respect to the 
Thompson metric,
\item[(4)] $\sigma_{g}(\omega; \mathbb{I})=I$, 
where $\mathbb{I}=(I,...,I)$. 
\end{itemize}
Moreover, $\sigma_{g}((1-\lambda,\lambda);A,B)$ is a
$\lambda$-weighted operator mean 
in Kubo and Ando's sense for all
$\lambda\in [0,1]$.
\end{thqt}

More generalizations are discussed in 
\cite{HarXiv1711.10170, HSW2018, P2016}.

\subsection{The Ando-Hiai inequality}
The Ando-Hiai inequality is one of the most important 
inequalities in the operator theory.

\begin{thqt}[The Ando-Hiai inequality \cite{{AH1994}}]
\label{thqt:AH}
Let $A,B\in \mathcal{PS}$ and $\lambda\in [0,1]$. If 
$A\sharp_{\lambda}B \leq I$ holds, then 
$A^{r}\sharp_{\lambda}B^{r} \leq I$
holds for all $r\geq 1$.
\end{thqt}

The Ando-Hiai inequality has been extended into 
the following two types.

\begin{thqt}[Extension of the Ando-Hiai inequality 1,
\cite{LL2014, LP2012, LY2013, Y2012}]
\label{thqt:exAH1}
Let $\mathbb{A}=(A_{1},...,A_{n})\in \mathcal{P}^{n}$, 
$\omega\in \Delta_{n}$ and $t\in (0,1]$.
Then the following hold.
\begin{itemize}
\item[(1)] 
$\Lambda(\omega; \mathbb{A})\leq I$ implies $\Lambda(\omega; \mathbb{A}^{r})\leq I$
for all $r\geq 1$,
\item[(2)] $P_{t}(\omega; \mathbb{A})\leq I$ 
implies $P_{\frac{t}{r}}(\omega; \mathbb{A}^{r})\leq I$
for all $r\geq 1$,
\item[(3)] $P_{-t}(\omega; \mathbb{A})\geq I$ 
implies $P_{-\frac{t}{r}}(\omega; \mathbb{A}^{r})\geq I$
for all $r\geq 1$,
\end{itemize}
where $\mathbb{A}^{r}=(A_{1}^{r},...,A_{n}^{r})$.
\end{thqt}

We remark that opposite inequalities 
of Theorems \ref{thqt:AH} and \ref{thqt:exAH1} (1)
hold because
$\Lambda(\omega; \mathbb{A})^{-1}=
\Lambda(\omega; \mathbb{A}^{-1})$
holds for all $\mathbb{A}\in \mathcal{P}^{n}$ 
and $\omega \in \Delta_{n}$,
where $\mathbb{A}^{-1}=(A_{1}^{-1},...,A_{n}^{-1})$.
Moreover, the Karcher mean is
a unique geometric mean satisfying 
all 10 properties stated in \cite{ALM2004}
and the property of Theorem \ref{thqt:exAH1} (1)
\cite{Y2012}.

We notice for Theorem \ref{thqt:exAH1}
(2) and (3) that different 
power means are appeared in each statement, 
more precisely, there are power means with 
different parameters. 
On the other hand
the Ando-Hiai inequality has been extended 
further to 
the following form.

\begin{thqt}[Extension of the Ando-Hiai inequality 2,
\cite{W2014}]\label{Ando-Hiai Wada}
Let $\sigma$ be an operator mean with a 
representing function $f$. 
Then the following are equivalent.
\begin{itemize}
\item[(1)] $f(x^{r})\leq f(x)^{r}$ holds for all 
$x\in (0,\infty)$ and $r\geq 1$,
\item[(2)] $\sigma(A,B)\leq I$ implies 
$\sigma(A^{r},B^{r})\leq I$ for all $A,B\in \mathcal{PS}$
and $r\geq 1$.
\end{itemize}
\end{thqt}

\section{Relations between generalized Karcher
 equation, relative operator entropy and 
operator means}

In this section, we shall give a relation between
the GKE and operator means.
First, we shall give a concrete form of 
an inverse function of the representing function 
of an operator mean which is derived from 
a given GKE.
Before showing results, we notice the following.
The representing function of an operator mean 
is defined only for two-variable operator means.
In this paper, we usually treat 
operator means of $n$ operators, and as 
a special case, we can treat operator means 
of two operators. Here we shall use the representing 
function of an operator mean as follows. 
Let $\sigma_{g}$ 
be a solution of \eqref{GKE}. Then 
for $\lambda\in [0,1]$, its representing
function $f_{\lambda}$ is defined by
$$ f_{\lambda}(x)=
\sigma_{g}((1-\lambda,\lambda); 1,x), $$
i.e., $f_{\lambda}(x)$ satisfies the following GKE:
\begin{equation}
 (1-\lambda)g\left(\frac{1}{f_{\lambda}(x)}\right)+
\lambda g\left(\frac{x}{f_{\lambda}(x)}\right)=0
\label{eq: GKE for representing function}
\end{equation}
for all $x>0$.
We note that $f_{1}(x)=x$ and $f_{0}(x)=1$ by 
\eqref{eq: GKE for representing function}.
Hence we can define $f_{\lambda}$ for all $\lambda\in 
[0,1]$.
Since $\sigma_{g}(1-\lambda, \lambda; A,B)$ is
an operator mean in Kubo and Ando's sense
by Theorem \ref{prop:GKE mean},
the representing function $f_{\lambda}$ is
an operator monotone function on $(0,\infty)$
and $f(1)=1$.

\begin{prop}[see also \cite{ParXiv1208.5603}]
\label{inverse function}
Let $g\in \mathcal{L}$. Then for each 
$\lambda\in (0,1)$,
the inverse of $f_{\lambda}$ in 
\eqref{eq: GKE for representing function} 
is given by
$$ f^{-1}_{\lambda}(x)=xg^{-1}\left(-\frac{1-\lambda}{\lambda}g\left(\frac{1}{x}\right)\right).$$
\end{prop}

\begin{proof}
Let $\sigma_{g}$ 
be an operator mean derived from the GKE.
Then for each $\lambda\in (0,1)$, 
$y=f_{\lambda}(x)$
satisfies the following equation
\begin{equation*}
(1-\lambda)g\left(\frac{1}{y}\right)
+\lambda g\left(\frac{x}{y}\right)=0. 
\end{equation*}
It is equivalent to 
$$ g\left(\frac{x}{y}\right)=-\frac{1-\lambda}{\lambda}
g\left(\frac{1}{y}\right), $$
and thus
$$ f^{-1}_{\lambda}(y)=x=
yg^{-1}\left(-\frac{1-\lambda}{\lambda}g\left(
\frac{1}{y}\right)\right).$$
The proof is completed.
\end{proof}

\noindent
{\bf Remark.}
If $f_{\lambda}$ is non-constant, then
by \eqref{eq: GKE for representing function},
% and Proposition \ref{inverse function}, 
we can get 
the range of $f_{\lambda}$ 
under a condition  
$\displaystyle g(0):=\lim_{x\searrow 0}g(x)=-\infty$ or 
$\displaystyle g(\infty):=\lim_{x\nearrow +\infty} =+\infty$
as follows:
Since $f_{\lambda}$ is operator monotone on 
$(0,\infty)$, the range of $f_{\lambda}$ is 
$\displaystyle (y_{0},y_{\infty}):=
(\lim_{x\searrow 0}f_{\lambda}(x),
\lim_{x\nearrow +\infty}f_{\lambda}(x))$. 
Hence it is enough to determine $y_{0}$
and $y_{\infty}$.
To obtain them, we shall divide the discussion into 
four cases:
(1) $g(0)=-\infty$ and $g(\infty)=+\infty$,
(2) $g(0)=-\infty$ and $g(\infty)<+\infty$,
(3) $g(0)>-\infty$ and $g(\infty)=+\infty$, and
(4) $g(0)>-\infty$ and $g(\infty)<+\infty$.
In fact, we can give examples of 
operator monotone functions $g_{i}$ $(i=1,2,3,4)$
which satisfy $g_{i}(1)=0$, $g_{i}'(1)=1$ and 
the above condition ($i$) $(i=1,2,3,4)$, 
respectively as follows:
$g_{1}(x)=\log x$,
$g_{2}(x)=1-\frac{1}{x}$, 
$g_{3}(x)=2(\sqrt{x}-1)$ and
$g_{4}(x)=2(1-\frac{2}{x+1})$. 
Here we shall give $y_{0}$ and 
$y_{\infty}$ in each case.

First of all, we shall discuss the two cases.
\begin{itemize}
\item[(a)] If $g(0)=-\infty$, then $y_{0}=0$.
In fact, assume $y_{0}>0$. Let $x\searrow 0$ in 
\eqref{eq: GKE for representing function}. Then
$y:=f_{\lambda}(x) \searrow y_{0}>0$ and 
$$ 0=(1-\lambda)g\left(\frac{1}{y_{0}}\right)+\lambda
g(0)=-\infty. $$
It is a contradiction. Hence $y_{0}=0$.
\item[(b)] If $g(\infty)=+\infty$, then $y_{\infty}=
+\infty$. In fact, assume $y_{\infty}<+\infty$. 
Let $x\nearrow +\infty$ in 
\eqref{eq: GKE for representing function}. 
Then $y\nearrow y_{\infty}<+\infty$ and
$$ 0=(1-\lambda)g\left(\frac{1}{y_{\infty}}\right)+
\lambda g(\infty)=+\infty. $$
It is a contradiction. Hence $y_{\infty}=+\infty$
\end{itemize}

Next, we shall give $y_{0}$ and $y_{\infty}$ in
cases (1) -- (4).
\begin{itemize}
\item[(1)] $y_{0}=0$ and $y_{\infty}=+\infty$ by 
 (a) and (b).
\item[(2)] $y_{0}=0$ by (a). We shall show 
$y_{\infty}=1/g^{-1}(-\frac{\lambda}{1-\lambda}g(\infty))
<+\infty$. 
Assume $y_{\infty}=+\infty$. 
Let $x\nearrow +\infty$ in
\eqref{eq: GKE for representing function}.
Then $y\nearrow y_{\infty}=+\infty$ and 
\begin{align*}
0  = (1-\lambda)g\left(\frac{1}{y}\right)+
\lambda g\left(\frac{x}{y}\right) 
& \leq 
(1-\lambda)g\left(\frac{1}{y}\right)+
\lambda g\left(\infty \right) \\
& \to 
(1-\lambda)g\left(0\right)+
\lambda g\left(\infty \right) =-\infty.
\end{align*}
It is a contradiction. Hence $y_{\infty}<+\infty$.
In this case, we can get 
$y_{\infty}$ as follows. 
By \eqref{eq: GKE for representing function},
we have
$ y=1/g^{-1}(-\frac{\lambda}{1-\lambda}g(\frac{x}{y}))$, and
$$ y_{\infty}=\lim_{x\nearrow \infty}
\frac{1}{g^{-1}(-\frac{\lambda}{1-\lambda}g(\frac{x}{y}))}=
\frac{1}{g^{-1}(-\frac{\lambda}{1-\lambda}g(\infty))}. $$

\item[(3)] $y_{\infty}=+\infty$ by (b). We shall show 
$y_{0}=1/g^{-1}(-\frac{\lambda}{1-\lambda}g(0))
>0$. 
Assume $y_{0}=0$. Let $x\searrow 0$ in
\eqref{eq: GKE for representing function}.
Then $y\searrow y_{0}=0$ and 
\begin{align*}
0  = (1-\lambda)g\left(\frac{1}{y}\right)+
\lambda g\left(\frac{x}{y}\right) 
& \geq 
(1-\lambda)g\left(\frac{1}{y}\right)+
\lambda g\left(0 \right) \\
& \to 
(1-\lambda)g\left(\infty \right)+
\lambda g\left(0 \right) =+\infty.
\end{align*}
It is a contradiction. Hence $y_{0}>0$.
In this case, we can get 
$y_{0}$ as follows. 
By \eqref{eq: GKE for representing function},
we have
$ y=1/g^{-1}(-\frac{\lambda}{1-\lambda}g(\frac{x}{y}))$, and
$$ y_{0}=\lim_{x\searrow 0}
\frac{1}{g^{-1}(-\frac{\lambda}{1-\lambda}g(\frac{x}{y}))}=
\frac{1}{g^{-1}(-\frac{\lambda}{1-\lambda}g(0))}. $$

\item[(4)] For this case, we cannot settle 
$y_{0}$ and $y_{\infty}$, generally. 
For example, if $y_{0}=0$, then
\eqref{eq: GKE for representing function} 
and monotonicity of $g$ imply
\begin{align*}
0 = (1-\lambda)g\left(\frac{1}{y}\right)+
\lambda g\left(\frac{x}{y}\right)
& \geq 
(1-\lambda)g\left(\frac{1}{y}\right)+
\lambda g\left(0\right)\\
& \to 
(1-\lambda)g\left(\infty \right)+
\lambda g\left(0\right) \quad 
\text{(as $x\searrow 0$)}.
\end{align*}
Hence  
if $(1-\lambda)g\left(\infty \right)+
\lambda g\left(0\right)>0$, then 
$y_{0}=1/g^{-1}(-\frac{\lambda}{1-\lambda}g(0))>0$
by the contraposition.
If $(1-\lambda)g\left(\infty \right)+
\lambda g\left(0\right)\leq 0$, then we cannot 
settle $y_{0}$.

If $y_{\infty}=+\infty$, then
\eqref{eq: GKE for representing function} implies
\begin{align*}
0 = (1-\lambda)g\left(\frac{1}{y}\right)+
\lambda g\left(\frac{x}{y}\right)
& \leq 
(1-\lambda)g\left(\frac{1}{y}\right)+
\lambda g\left(\infty \right)\\
& \to 
(1-\lambda)g\left(0 \right)+
\lambda g\left(\infty \right)\quad
\text{(as $x\nearrow +\infty$)}.
\end{align*}
Hence if $(1-\lambda)g\left(0 \right)+
\lambda g\left(\infty \right)<0$, then 
$y_{\infty}=
1/g^{-1}(-\frac{\lambda}{1-\lambda}g(\infty))<+\infty$
by the contraposition.
If $(1-\lambda)g\left(0 \right)+
\lambda g\left(\infty \right)\geq 0$, then we cannot 
settle $y_{\infty}$.

More precisely, 
the range of $f_{\lambda}$ 
depends on $\lambda$ in this case. 
For example, let $g_{4}(x)=2(1-\frac{2}{x+1})$ and 
$\lambda=\frac{1}{2}$. 
Then by \eqref{eq: GKE for representing function},
we have $y=f_{\frac{1}{2}}(x)=\sqrt{x}$, 
and hence $y_{0}=0$ and 
$y_{\infty}=\infty$. On the other hand, let 
$\lambda=\frac{1}{4}$. Then we have
$y=f_{\frac{1}{4}}(x)=
\frac{1}{4}\{(1-x)+\sqrt{(1-x)^{2}+16x}\}$, and
$y_{0}=\frac{1}{2}>0$ and $y_{\infty}=2<+\infty$.
\end{itemize}

\begin{prop}\label{return to g}
Let $g\in \mathcal{L}$
and $f_{\lambda}$ be the representing function 
of $\sigma_{g}$. Then
$$ 
\left. \frac{\partial}{\partial \lambda}f_{\lambda}(x)
 \right|_{\lambda=0}
%=
%\lim_{\lambda\searrow 0} 
% \frac{\partial}{\partial \lambda}f_{\lambda}(x)
=g(x).$$
\end{prop}

For the Karcher mean case, 
$f_{\lambda}(x)=x^{\lambda}$ and 
$g(x)=\frac{\partial}{\partial \lambda}x^{\lambda}
|_{\lambda=0}=\log x$. We note that $g(x)=\log x$ is a representing
function of the relative operator entropy 
\cite{F1992}. In fact 
let $A,B\in \mathcal{P}$. Then the relative 
operator entropy $S(A|B)$ is defined by
$$ S(A|B)=A^{\frac{1}{2}}\log (
A^{\frac{-1}{2}}BA^{\frac{-1}{2}})A^{\frac{1}{2}}. $$
For the power mean case, 
$f_{\lambda}(x)=
[1-\lambda+\lambda x^{t}]^{\frac{1}{t}}$ and 
$g(x)=
\frac{\partial}{\partial \lambda}
[1-\lambda+\lambda x^{t}]^{\frac{1}{t}}
|_{\lambda=0}
=\frac{x^{t}-1}{t}$. We note that 
 $g(x)=\frac{x^{t}-1}{t}$ is a representing
function of the Tsallis relative operator entropy
\cite{YKF2005}. 
In fact 
let $A,B\in \mathcal{P}$. Then the Tsallis 
relative operator entropy $T_{t}(A|B)$ is defined by
$$ T_{t}(A|B)=A^{\frac{1}{2}}\frac{
(A^{\frac{-1}{2}}BA^{\frac{-1}{2}})^{t}-I}{t}
A^{\frac{1}{2}}
=
\frac{A\sharp_{t}B-A}{t}. $$
So relative operator entropy is closely related to
the GKE and operator means.

To prove Proposition \ref{return to g},
we shall prepare the following lemma.

\begin{lem}\label{lem:limit}
Let $g\in \mathcal{L}$
and $f_{\lambda}$ be the representing function 
of $\sigma_{g}$. Then
$f_{1-\lambda}(x)=xf_{\lambda}(1/x)$ holds
for all $\lambda\in [0,1]$ and $x\in (0,\infty)$.
Moreover, $\lim_{\lambda\to 0}f_{\lambda}(x)
=1$ holds for all $x\in (0,\infty)$.
\end{lem}

\begin{proof}
Since $f_{\lambda}$ is the representing function 
of $\sigma_{g}$, the following GKE holds
for all $x\in (0,\infty)$.
$$ (1-\lambda) g\left( \frac{1}{f_{\lambda}(x)}\right)+
\lambda g\left( \frac{x}{f_{\lambda}(x)}\right)=0.$$
By replacing $x$ into $1/x$, it is equivalent to
$$ (1-\lambda) g\left( 
\frac{x}{xf_{\lambda}(1/x)}\right)+
	\lambda g\left( \frac{1}{xf_{\lambda}(1/x)}\right)=0.$$
Hence a function $xf_{\lambda}(1/x)$ should be 
the representing function of $\sigma_{g}$ with 
a probability vector $(\lambda, 1-\lambda)$, i.e.,
$f_{1-\lambda}(x)=xf_{\lambda}(1/x)$.

Next, we shall show $\lim_{\lambda\to 0}f_{\lambda}(x)
=1$ holds for all $x\in (0,\infty)$.
By Proposition \ref{inverse function}, 
we have
$$ \lim_{\lambda\to 1}f_{\lambda}^{-1}(x)=
xg^{-1}(0)=x,$$
where the last equality follows from 
$g(1)=0$.
By \cite[Lemma 3.6]{HarXiv1711.10170},
$f_{\lambda}$ is equicontinuous
on any bounded interval in $(0,\infty)$, 
and we have
$$  \lim_{\lambda\to 1}f_{\lambda}(x)=
\lim_{\lambda,\mu\to 1}f_{\lambda}(f_{\mu}^{-1}(x))
=
\lim_{\lambda\to 1}f_{\lambda}(f_{\lambda}^{-1}(x))
=x.$$
Therefore we have
$$ \lim_{\lambda\to 0}f_{\lambda}(x)=
\lim_{\lambda\to 0}
xf_{1-\lambda}(1/x)=1.$$
\end{proof}

\begin{proof}[Proof of Proposition \ref{return to g}]
First of all, $g\in \mathcal{L}$ is at least twice
differentiable since $g$ is an 
operator monotone function.
By $g\in \mathcal{L}$ and Proposition \ref{inverse function}, the 
representing function $f_{\lambda}$ is 
at least twice differentiable on 
$\lambda\in (0,1)$.
By differentiating 
\eqref{eq: GKE for representing function}
both sides in $\lambda$, we have
\begin{align*}
-g\left(\frac{1}{f_{\lambda}(x)}\right)
& +(1-\lambda)
g'\left(\frac{1}{f_{\lambda}(x)}\right)
\left(-\frac{1}{f_{\lambda}(x)^{2}}\right)
\frac{\partial}{\partial \lambda}f_{\lambda}(x) \\
 & +  
g\left(\frac{x}{f_{\lambda}(x)}\right)
+\lambda g'\left(\frac{x}{f_{\lambda}(x)}\right)
\left(\frac{-x}{f_{\lambda}(x)^{2}}\right)
\frac{\partial}{\partial \lambda}f_{\lambda}(x) =0.
\end{align*}
Here we have 
$\lim_{\lambda\to 0} f_{\lambda}(x)\to 1$
by Lemma \ref{lem:limit}, 
$g(1)=0$ and $g'(1)=1$, we have
$$ \left. \frac{\partial}{\partial \lambda}f_{\lambda}(x)
\right|_{\lambda=0}=g(x). $$
\end{proof}

\section{The Ando-Hiai inequalities for the solution of 
the GKE}

In this section, we shall show extensions of 
the Ando-Hiai inequalities. To prove them, the following 
result is very important.

\begin{thm}\label{thm1}
Let $g\in \mathcal{L}$, $\mathbb{A}=
(A_{1},...,A_{n})\in \mathcal{P}^{n}$ and 
$\omega=(w_{1},...,w_{n})\in \Delta_{n}$. 
Then the following hold.
\begin{itemize}
\item[(1)] $\sum_{i=1}^{n}w_{i}g(A_{i})\geq 0$ implies
$ \sigma_{g}(\omega; \mathbb{A})\geq I,$ and 
\item[(2)] $\sum_{i=1}^{n}w_{i}g(A_{i})\leq 0$ implies
$ \sigma_{g}(\omega; \mathbb{A})\leq I.$
\end{itemize}
\end{thm}

To prove Theorem \ref{thm1}, we shall prepare 
the following property of $\sigma_{g}$. 

\begin{lem}\label{lem:upper-lower bounds}
Let $g\in \mathcal{L}$, 
$\omega=(w_{1},...,w_{n})\in \Delta_{n}$ and 
$\mathbb{A}=(A_{1},...,A_{n})\in \mathcal{P}^{n}$. 
Then
$$ \left[ \sum_{i=1}^{n}w_{i}A^{-1}_{i}\right]^{-1}
\leq \sigma_{g}(\omega; \mathbb{A})
\leq \sum_{i=1}^{n}w_{i}A_{i}. $$
\end{lem}

\begin{proof}
We note that for each $g\in \mathcal{L}$, 
$$ 1-x^{-1}\leq g(x)\leq x-1 $$
holds for all $x\in (0,\infty)$ \cite[(18)]{P2016}.
Let $X=\sigma_{g}(\omega; \mathbb{A})$. Then we have
$$0 = \sum_{i=1}^{n}w_{i}g(X^{\frac{-1}{2}}A_{i}
X^{\frac{-1}{2}})
 \leq
\sum_{i=1}^{n}w_{i}(X^{\frac{-1}{2}}A_{i}X^{\frac{-1}{2}}-I), $$
i.e., $X\leq \sum_{i=1}^{n} w_{i}A_{i}$.
The rest of the part can be shown in the same way 
and using 
$g(x)\geq 1-x^{-1}$.
\end{proof}

\begin{proof}[Proof of Theorem \ref{thm1}]
Proof of (1). 
Assume that $\sum_{i=1}^{n}w_{i}g(A_{i})\geq 0$ holds.
We note that $g$ is invertible because 
it is an operator monotone function. Moreover
by $g(1)=0$, there exists $X\leq I$ such that 
$$ \alpha \sum_{i=1}^{n}w_{i}g(A_{i})+
(1-\alpha)g(X)=0 $$
for a sufficiently small $\alpha\in (0,1)$.
Since we can consider the above equation
as a GKE, we have 
$$
I  = \sigma_{g}\left( \left(\alpha \omega, 
1-\alpha \right);
(\mathbb{A},X)\right),$$
where $(\alpha \omega, 1-\alpha)=
(\alpha w_{1},...,\alpha w_{n},1-\alpha)\in 
\Delta_{n+1}$
and $(\mathbb{A},X)=(A_{1},...,A_{n},X)\in 
\mathcal{P}^{n+1}$.
Here we define an operator sequence $\{X_{k}\}
\subset \mathcal{P}$ by
$$ X_{0}=I,\ X_{k+1}=\sigma_{g}\left( 
\left(\alpha \omega,
1-\alpha \right); (\mathbb{A}, X_{k})\right). $$
Then 
\begin{align*}
X_{0}= I & = \sigma_{g}\left( 
\left(\alpha \omega, 1-\alpha \right);
(\mathbb{A},X)\right)\\
& \leq 
\sigma_{g}\left( \left(\alpha \omega, 
1-\alpha \right);
(\mathbb{A},I)\right)=X_{1} \\
& \leq 
\sigma_{g}\left( 
\left(\alpha \omega, 1-\alpha \right);
(\mathbb{A},X_{1})\right)=X_{2} 
 \leq \cdots \leq X_{n}, 
\end{align*}
where the inequalities hold by 
operator monotonicity of $\sigma_{g}$, i.e., 
Theorem \ref{prop:GKE mean} (1).
By Lemma \ref{lem:upper-lower bounds}, we have 
$$
X_{k}  \leq 
\sigma_{g}\left( 
\left(\alpha \omega, 1-\alpha \right);
(\mathbb{A},X_{k})\right)
 \leq 
\alpha \sum_{i=1}^{n}w_{i}A_{i}+(1-\alpha) X_{k}, 
$$
and we have $X_{k}\leq \sum_{i=1}^{n}w_{i}A_{i}$ for all
$k=1,2,...$. Hence 
there exists a unique limit point 
$\lim_{k\to \infty}X_{k}=X_{\infty}\in \mathcal{P}$.
It satisfies 
$$ X_{\infty} =
\sigma_{g}\left( \left(\alpha \omega, 
1-\alpha \right);
(\mathbb{A},X_{\infty})\right), $$
and then we have
$$ \sum_{i=1}^{n}w_{i}g(X_{\infty}^{\frac{-1}{2}}A_{i}
X_{\infty}^{\frac{-1}{2}})=0, $$
that is, 
$$ I\leq X_{\infty}=\sigma_{g}\left( \omega;
\mathbb{A}\right). $$

Proof of (2) is shown in the same way by using 
$[\sum_{i=1}^{n}w_{i}A_{i}^{-1}]^{-1}\leq 
\sigma_{g}(\omega; \mathbb{A})$. 
\end{proof}

Using Theorem \ref{thm1}, we can get an elementary 
property of the solution of the GKE.

\begin{thm}\label{monotonicity of g}
Let $f,g\in \mathcal{L}$. Then $g(x)\leq f(x)$ holds for
all $x\in (0,\infty)$ if and only if 
$\sigma_{g}(\omega; \mathbb{A})\leq 
\sigma_{f}(\omega; \mathbb{A})$ holds for all 
$\omega\in \Delta_{n}$ and $\mathbb{A}\in 
\mathcal{P}^{n}$.
\end{thm}

\begin{proof}
Proof of $(\Longrightarrow)$.
Let $\omega=(w_{1},...,w_{n})\in \Delta_{n}$,  
$\mathbb{A}=(A_{1},...,A_{n})\in \mathcal{P}^{n}$
and $X=\sigma_{g}(\omega;\mathbb{A})$.
Assume that $g(x)\leq f(x)$ holds for
all $x\in (0,\infty)$. Then 
$$ 0=\sum_{i=1}^{n}w_{i}g(X^{\frac{-1}{2}}A_{i}
X^{\frac{-1}{2}})\leq 
\sum_{i=1}^{n}w_{i}f(X^{\frac{-1}{2}}A_{i}
X^{\frac{-1}{2}}). $$
By Theorem \ref{thm1} (1), we have
$I\leq \sigma_{f}(\omega; X^{\frac{-1}{2}}\mathbb{A}
X^{\frac{-1}{2}})=
X^{\frac{-1}{2}} \sigma_{f}(\omega; \mathbb{A})
X^{\frac{-1}{2}}, $ i.e., 
$$ \sigma_{g}(\omega;\mathbb{A})=X\leq 
\sigma_{f}(\omega; \mathbb{A}). $$

\medskip

Proof of $(\Longleftarrow)$.
It is enough to consider the two operators case.
For $\lambda\in [0,1]$, let $r_{g,\lambda}$ 
and $r_{f,\lambda}$
are representing functions of 
$\lambda$-weighted operator means 
$\sigma_{g}$ and $\sigma_{f}$,
respectively.
Then $r_{g,\lambda}(x)\leq r_{f,\lambda}(x)$ holds for all 
$x\in (0,\infty)$ and $\lambda\in [0,1]$, 
and we have
$$ \frac{r_{g,\lambda}(x)-1}{\lambda} 
\leq \frac{r_{f,\lambda}(x)-1}{\lambda} $$
holds for all $x\in (0,\infty)$ and $\lambda\in (0,1]$.
Let $\lambda\searrow 0$. Then we have $g(x)\leq f(x)$
by Proposition \ref{return to g}.
\end{proof}

Here, we shall show extensions of the Ando-Hiai 
inequality.

\begin{thm}[Extension of the 
Ando-Hiai inequality, 1]\label{thm2}
Let $g\in \mathcal{L}$, $\mathbb{A}\in \mathcal{P}^{n}$
and $\omega \in \Delta_{n}$. If 
$\sigma_{g}(\omega;\mathbb{A})\leq I$
holds, then 
$ \sigma_{g_{p}}(\omega; \mathbb{A}^{p})\leq I$
holds for all $p\geq 1$, where 
$g_{p}(x):=pg(x^{1/p})$.
Moreover the representing function of the
$\lambda$-weighted operator mean
$ \sigma_{g_{p}}$ is $f_{p,\lambda}(x):=
f_{\lambda}(x^{1/p})^{p}$ for all $\lambda\in [0,1]$,
where $f_{\lambda}$ is the representing function of 
the $\lambda$-weighted operator mean $\sigma_{g}$.
\end{thm}

We notice that $g_{p}(x)=pg(x^{1/p})\in \mathcal{L}$
for all $p\geq 1$.

By putting $g(x)=\log x$ in Theorem \ref{thm2},
$\sigma_{g}$ coincides with the Karcher mean.
Then we have Theorem \ref{thqt:exAH1} (1).
Moreover put $g(x)=\frac{x^{t}-1}{t}$ in 
Theorem \ref{thm2},
$\sigma_{g}$ coincides with the power mean.
Then we have Theorem \ref{thqt:exAH1} (2).

\begin{proof}[Proof of Theorem \ref{thm2}]
Let $X=\sigma_{g}(\omega; \mathbb{A})\leq I$.
For $p\geq 1$, we have
$$0  = 
\sum_{i=1}^{n}w_{i}g(X^{\frac{-1}{2}}A_{i}X^{\frac{-1}{2}}) 
%=
%\sum_{i=1}^{n}w_{i}g((X^{\frac{-1}{2}}A_{i}X^{\frac{-1}{2}}
%)^{\frac{p}{p}})
 \geq 
\sum_{i=1}^{n}w_{i}g((X^{\frac{-1}{2}}A_{i}^{p}X^{\frac{-1}{2}}
)^{\frac{1}{p}})
$$
by Hansen's
inequality \cite{H1979}.
Hence 
$$
0 \geq 
\sum_{i=1}^{n}w_{i}pg((X^{\frac{-1}{2}}A_{i}^{p}X^{\frac{-1}{2}})^{\frac{1}{p}})
=\sum_{i=1}^{n}w_{i}g_{p}(X^{\frac{-1}{2}}A_{i}^{p}X^{\frac{-1}{2}}),
$$
and we have $\sigma_{g_{p}}(\omega; 
X^{\frac{-1}{2}}\mathbb{A}^{p}X^{\frac{-1}{2}})\leq I$
by Theorem \ref{thm1} (2), i.e., 
$$ \sigma_{g_{p}}(\omega; \mathbb{A}^{p})\leq X\leq I$$
for all $p\geq 1$.
%Applying the same way to $ \sigma_{g_{p}}(\omega; \mathbb{A}^{p})\leq I$, we have
%$ \sigma_{g_{pp'}}(\omega; \mathbb{A}^{pp'})\leq I$
%for $p'\in [1,2]$ and $pp'\in [1,4]$.
%Repeating this method, we have 
%$ \sigma_{g_{p}}(\omega; \mathbb{A}^{p})\leq I$
%for all $p\geq 1$.
 
Let $f_{\lambda}$ and $f_{p, \lambda}$ be 
representing functions of 
$\lambda$-weighted operator means $\sigma_{g}$
and $\sigma_{g_{p}}$, respectively.
We note that the inverse function of 
$g_{p}(x)=pg(x^{1/p})$ is
$\{g^{-1}(\frac{x}{p})\}^{p}$. Hence 
by Proposition \ref{inverse function}, we have 
\begin{align*}
f^{-1}_{p,\lambda}(x) & = 
xg_{p}^{-1}\left( -\frac{1-\lambda}{\lambda}
g_{p}\left(\frac{1}{x}\right)\right)\\
& =
x\left\{ g^{-1}\left( -\frac{1-\lambda}{p\lambda}
\cdot p g\left(\frac{1}{x^{1/p}}\right)\right)\right\}^{p}\\
& =
\left\{ x^{\frac{1}{p}} g^{-1}
\left(-\frac{1-\lambda}{\lambda}
g\left(\frac{1}{x^{1/p}}\right)\right)\right\}^{p}
 =
f^{-1}_{\lambda}(x^{1/p})^{p}.
\end{align*}
Therefore $f_{p,\lambda}(x)=f_{\lambda}(x^{1/p})^{p}$.
\end{proof}

We can prove for the opposite inequalities 
in Theorem \ref{thm2} in the same way.

\medskip
\noindent
{\bf Theorem \ref{thm2}'.}
{\it Let $g\in \mathcal{L}$, $\mathbb{A}\in \mathcal{P}^{n}$
and $\omega \in \Delta_{n}$. If 
$\sigma_{g}(\omega;\mathbb{A})\geq I$
holds, then 
$ \sigma_{g_{p}}(\omega; \mathbb{A}^{p})\geq I$%
holds for all $p\geq 1$, where 
$g_{p}(x):=pg(x^{1/p})$.
Moreover the representing function of the
$\lambda$-weighted operator mean
$ \sigma_{g_{p}}$ is $f_{p,\lambda}(x):=
f_{\lambda}(x^{1/p})^{p}$ for all $\lambda\in [0,1]$,
where $f_{\lambda}$ is the representing function of 
the $\lambda$-weighted operator mean $\sigma_{g}$.}

\medskip

\begin{thm}[Extension of the 
Ando-Hiai inequality, 2]
\label{thm: extensionAH-wada}
Let $g\in \mathcal{L}$.
Assume $f_{\lambda}$ is the representing function of 
the $\lambda$-weighted operator mean 
$\sigma_{g}$. 
Then the following are equivalent.
\begin{itemize}
\item[(1)] $f_{p,\lambda}(x)=f_{\lambda}(x^{1/p})^{p}\leq 
f_{\lambda}(x)$ holds 
for all $p\geq 1$, $\lambda\in [0,1]$ and 
$x\in (0,\infty)$,
\item[(2)] $g_{p}(x)=pg(x^{\frac{1}{p}})\leq g(x)$ 
holds for all $p\geq 1$ and 
$x\in (0,\infty)$,
\item[(3)] $\sigma_{g}(\omega;\mathbb{A})\geq I$
implies $\sigma_{g}(\omega;\mathbb{A}^{p})\geq I$
for all $\omega\in \Delta_{n}$, $\mathbb{A}\in 
\mathcal{P}^{n}$ and $p\geq 1$.
\end{itemize}
\end{thm}

For the case of two operators, Theorem 
\ref{thm: extensionAH-wada} coincides with
the opposite inequalities of  
Theorem \ref{Ando-Hiai Wada}
(it was shown in \cite{W2014}).
Moreover, we can obtain a property of 
operator monotone functions in $\mathcal{L}$
in (2) of the above theorem
but it is not given in Theorem \ref{Ando-Hiai Wada}.
%For example, 
%for $r\in (0,1]$, let $g(x)=\frac{x^{r}-1}{r}$.
%It is a representing function of the Tsallis relative operator
%entropy. Then we have
%$f_{\lambda}(x)=[(1-\lambda)+\lambda x^{r}]^{1/r}$
%as a solution of 
%\eqref{eq: GKE for representing function}.
%It is known that $f_{\lambda}$ is increasing
%for $r\in (0,1]$, and it satisfies (1).
%Hence $g(x)$ satisfies (2).

\begin{proof}
Proof of (1) $\Longrightarrow$ (2).
First, the assumption is equivalent to 
$f_{\lambda}(x)^{p}\leq f_{\lambda}(x^{p})$ for
all $x\in (0,\infty)$.
Since $1+p(x-1)\leq x^{p}$ holds for all $p\geq 1$
and $x\in (0,\infty)$,
we have
$$ p\left(\frac{f_{\lambda}(x)-1}{\lambda}\right)
\leq \frac{f_{\lambda}(x)^{p}-1}{\lambda}
\leq 
\frac{f_{\lambda}(x^{p})-1}{\lambda}
$$
holds for all $p\geq 1$, $\lambda\in (0,1]$
and $x\in (0,\infty)$.
By taking a limit $\lambda \searrow 0$, we have 
$pg(x)\leq g(x^{p}) $ by Proposition \ref{return to g},
i.e., $g_{p}(x)=pg(x^{\frac{1}{p}})\leq g(x)$ for all $p\geq 1$
and $x\in (0, \infty)$.

\medskip

Proof of (2) $\Longrightarrow$ (3).
Let $X=\sigma_{g}(\omega; \mathbb{A})\geq I$.
For $p\in [1,2]$, 
\begin{align*}
0  = 
\sum_{i=1}^{n}w_{i}pg(X^{\frac{-1}{2}}
A_{i}X^{\frac{-1}{2}}) 
& \leq 
\sum_{i=1}^{n}w_{i}g((X^{\frac{-1}{2}}
A_{i}X^{\frac{-1}{2}})^{p}) \qquad (\text{by (2)})\\
& \leq 
\sum_{i=1}^{n}w_{i}g(X^{\frac{-1}{2}}
A_{i}^{p}X^{\frac{-1}{2}}),
\end{align*}
where the last inequality holds by 
Hansen and Pedersen's 
inequality \cite{HP1982}. 
Hence by Theorem \ref{thm1} (1), we have
$I\leq \sigma_{g}(\omega; X^{\frac{-1}{2}}
\mathbb{A}^{p}X^{\frac{-1}{2}})$, i.e., 
$$ I\leq X\leq \sigma_{g}(\omega; \mathbb{A}^{p}).$$
Applying the same way to 
$ I\leq \sigma_{g}(\omega; \mathbb{A}^{p})$, 
we have
$ I\leq \sigma_{g}(\omega; \mathbb{A}^{pp'})$
for $p'\in [1,2]$ and $pp'\in [1,4]$.
Repeating this method, we have 
$ I\leq \sigma_{g}(\omega; \mathbb{A}^{p})$
for all $p\geq 1$.

\medskip

Proof of (3) $\Longrightarrow$ (1) is shown in 
\cite{W2014}.
\end{proof}

The opposite inequalities in 
Theorem \ref{thm: extensionAH-wada} 
can be shown in a similar way. But the proof is 
a little bit 
different, we shall give a different part of the proof.

\medskip

\noindent
{\bf Theorem \ref{thm: extensionAH-wada}'.}
{\it Let $g\in \mathcal{L}$.
Assume $f_{\lambda}$ is a representing function of 
a $\lambda$-weighted operator mean 
$\sigma_{g}$. 
Then the following are equivalent.}
\begin{itemize}
\item[(1)] $f_{p,\lambda}(x)=f_{\lambda}(x^{1/p})^{p}\geq 
f_{\lambda}(x)$ {\it holds 
for all $p\geq 1$, $\lambda\in [0,1]$ and 
$x\in (0,\infty)$,}
\item[(2)] {\it $g_{p}(x)=pg(x^{1/p})\geq g(x)$ for all 
$p\geq 1$ and $x\in (0,\infty)$,}
\item[(3)] {\it $\sigma_{g}(\omega;\mathbb{A})\leq I$
implies $\sigma_{g}(\omega;\mathbb{A}^{p})\leq I$
for all $\omega\in \Delta_{n}$, $\mathbb{A}\in 
\mathcal{P}^{n}$ and $p\geq 1$.}
\end{itemize}

It is just an extension of Theorem \ref{Ando-Hiai Wada}.

\begin{proof}
Proof of (1) $\Longrightarrow$ (2).

Since $f_{0}(x^{1/p})^p=f_{0}(x)=1$ always holds 
by \eqref{eq: GKE for representing function},
we may assume $\lambda>0$.
The assumption is equivalent to 
$f_{\lambda}(x)^{p}\geq f_{\lambda}(x^{p})$.
Moreover it is equivalent to 
$$ f^{-1}_{\lambda}(x)^{p}\leq f_{\lambda}^{-1}(x^{p}). $$ 
By Proposition \ref{inverse function}, we have
$$ x^{p}g^{-1}\left(-\frac{1-\lambda}{\lambda}
g\left(\frac{1}{x}\right)\right)^{p}\leq 
x^{p}g^{-1}\left(-\frac{1-\lambda}{\lambda}
g\left(\frac{1}{x^{p}}\right)\right),
$$
i.e., 
$$ g^{-1}\left(-\frac{1-\lambda}{\lambda}
g\left(\frac{1}{x}\right)\right)^{p}\leq 
g^{-1}\left(-\frac{1-\lambda}{\lambda}
g\left(\frac{1}{x^{p}}\right)\right)
$$
holds for all $\lambda\in (0,1)$, $x\in (0,\infty)$ 
and $p\geq 1$.
Put $\alpha=\frac{1-\lambda}{\lambda}>0$ and 
replacing $\frac{1}{x}$ into $x$. Then it 
is equivalent to 
$$ g^{-1}\left(-\alpha
g\left(x\right)\right)^{p}\leq 
g^{-1}\left(-\alpha 
g\left(x^{p}\right)\right)
$$
holds for all $\alpha>0$, $x>0$ and $p\geq 1$.
Hence we have 
$$ \lim_{\alpha \searrow 0}
\frac{g^{-1}(-\alpha g(x))^{p}-1}{-\alpha}
\geq 
\lim_{\alpha \searrow 0}
\frac{g^{-1}(-\alpha g(x^{p}))-1}{-\alpha}.
$$
By L'Hospital's Rule, 
\begin{align*}
\lim_{\alpha\searrow 0}
 \frac{pg^{-1}(-\alpha g(x))^{p-1}}{-1}
\cdot 
\frac{-g(x)}{g'(g^{-1}(-\alpha g(x)))} 
 \geq 
\lim_{\alpha\searrow 0}
\frac{g(x^{p})}{g'(g^{-1}(-\alpha g(x^{p})))}.
\end{align*}
Therefore we have $pg(x)\geq g(x^{p})$ 
holds for all $x\in (0,\infty)$ and $p>1$.

\medskip

Proofs of (2) $\Longrightarrow$ (3) and
(3) $\Longrightarrow$ (1) are 
almost the same as the proof of 
Theorem \ref{thm: extensionAH-wada}.
\end{proof}

\section{A problem -- norm inequalities}
In this section, we shall discuss a norm 
inequality which is related to 
Theorem \ref{thm1}.
For the Karcher mean case, 
the following norm inequality holds:
Let $\mathbb{A}=(A_{1},...,A_{n})\in \mathcal{P}^{n}$ and 
$\omega=(w_{1},...,w_{n})\in \Delta_{n}$.
Then 
$$ \| \Lambda(\omega; \mathbb{A})\| \leq
\| \exp(\sum_{i=1}^{n} w_{i}\log A_{i})\| $$
holds for all unitarily invariant norms.
It was obtained in \cite{HP2012}.
If $\{A_{1},...,A_{n}\}$ is commuting, then 
the equality holds since 
$\Lambda(\omega; \mathbb{A})=
 \exp(\sum_{i=1}^{n} w_{i}\log A_{i})$ holds.

On the other hand, for the power mean case,
the following norm inequality also holds.
Let $t\in (0,1]$, $\mathbb{A}=(A_{1},...,A_{n})\in 
\mathcal{P}^{n}$ and 
$\omega=(w_{1},...,w_{n})\in \Delta_{n}$.
Then 
$$ \| P_{t}(\omega; \mathbb{A})\|_{p} \leq
\| \left[ \sum_{i=1}^{n}w_{i}A_{i}^{t}
\right]^{\frac{1}{t}}\|_{p} $$
holds for Shatten $p$-norms for all $p\geq 1$.
It was obtained in \cite{BLY2016, DDF2017, LY2013}.
If $\{A_{1},...,A_{n}\}$ is commuting, then 
the equality holds since 
$P_{t}(\omega; \mathbb{A})=
 \left[ \sum_{i=1}^{n}w_{i}A_{i}^{t}\right]^{\frac{1}{t}}$ holds.
We remark that it has not been known whether
the above norm inequality holds for all 
unitarily invariant 
norms or not.

In the operator norm case, 
the above inequalities follow from 
Theorem \ref{thm1} for $g(x)=\log x$ and
$g(x)=\frac{x^{t}-1}{t}$.
So one might expect that the following conjecture 
is true.

\medskip

\noindent
{\bf Conjecture.}
{\it Let $g\in \mathcal{L}$, 
$\mathbb{A}=(A_{1},...,A_{n})\in \mathcal{P}^{n}$ and 
$\omega=(w_{1},...,w_{n})\in \Delta_{n}$. Then}
$$ \|\sigma_{g}(\omega; \mathbb{A})\|
\leq \| g^{-1}\left( \sum_{i=1}^{n}w_{i}g(A_{i})\right)\| $$
{\it holds for the operator norm.}

\medskip

The author thinks that the above conjecture 
is not true in general because of the following 
reasons.
\begin{itemize}
\item[(1)] If an $n$-variable function
$f(x_{1},...,x_{n})=
g^{-1}\left( \sum_{i=1}^{n}w_{i}g(x_{i})\right)$
is positively homogeneous, then we can prove 
that the above conjecture is true. 
But $f(x_{1},...,x_{n})$ is not
positively homogeneous, in general.
\item[(2)] For commuting two operators $A$ and 
$B$, 
$$ \sigma_{g}((1-\lambda,\lambda); A,B)=
g^{-1}\left( (1-\lambda)g(A)+\lambda g(B)\right)$$
does not hold in generally. It follows from the following 
result.
\end{itemize}

\begin{thm}
Let $\{\sigma_{\lambda}\}_{\lambda\in [0,1]}$ be a family of  
$\lambda$-weighted 
operator means satisfying the following inequality
$$ \left[ (1-\lambda)A^{-1}+\lambda B^{-1}\right]^{-1}
\leq \sigma_{\lambda} (A,B) 
\leq (1-\lambda)A+\lambda B$$
for all $\lambda\in [0,1]$ and $A,B\in 
\mathcal{P}$.
Then for each $A,B\in \mathcal{P}$ such that
$AB=BA$, the following are equivalent.
\begin{itemize}
\item[(1)] For each $\lambda\in [0,1]$,
$\sigma_{\lambda} (A,B)$ is a 
$\lambda$-weighted 
 power mean, i.e.,
$$ \sigma_{\lambda} (A,B)=
A^{\frac{1}{2}}\left[ (1-\lambda)I+
\lambda ( A^{\frac{-1}{2}}BA^{\frac{-1}{2}})^{t}
\right]^{\frac{1}{t}}A^{\frac{1}{2}},$$
\item[(2)] There exists a real-valued function $g$ on
$(0,\infty)$ such that
$$\sigma_{\lambda}(A,B)=g^{-1}\left((1-\lambda) g(A) +
\lambda g(B) \right)$$
for all $\lambda\in [0,1]$.
\end{itemize}
\end{thm}

\begin{proof}
Proof of (1) $\Longrightarrow$ (2) is obvious
by taking $g(x)=\frac{x^{t}-1}{t}$.
So we shall prove (2) $\Longrightarrow$ (1).
Assume that (2) holds. Then 
$\{\sigma_{\lambda}\}_{\lambda\in [0,1]}$ should 
have the interpolatinal property 
\cite[Theorem 4]{UYY2017}.
Moreover since $\sigma_{\lambda}$ 
is an operator mean, 
it should be the power mean by 
\cite[Theorem 6]{UYY2017}.
\end{proof}

\medskip
\noindent
{\bf Acknowledgment.} The author wishes 
his thanks to the anonymous referees
for careful reading the paper and 
several helpful comments 
to improve the paper.

\end{document}